\documentclass[letterpaper, 10 pt, conference]{ieeeconf}  

\IEEEoverridecommandlockouts                              

\makeatletter
\let\NAT@parse\undefined
\makeatother
\usepackage[bookmarks=false,pdftex,pagebackref=false]{hyperref}

\hypersetup{									
	pdfdisplaydoctitle=true,					
	bookmarksnumbered=true,						
	colorlinks,
	citecolor=NavyBlue,
	bookmarksopen=true,
	pdfauthor={Pedro Santos, Joel Reis, Paulo Oliveira, and Carlos Silvestre},
	pdftitle={Geometric Reduced-Attitude Tracking under a Time-Varying Conic Constraint via Smooth Reference-Shaping},
	pdfsubject={},
	pdfkeywords={Attitude tracking, constrained control, geometric control}
}
\usepackage{graphics} 
\usepackage{amsmath,amsfonts,mathtools,amsthm} 
\usepackage{amssymb}  
\usepackage{color}
\usepackage[dvipsnames]{xcolor}
\let\b\mathbf
\let\bs\boldsymbol
\newcommand{\overbar}[1]{\mkern 1.5mu\overline{\mkern-1.5mu#1\mkern-1.5mu}\mkern 1.5mu}
\usepackage{tikz}
\usetikzlibrary{arrows.meta, positioning, shapes, calc}
\usepackage{cite}
\usepackage{balance}

\newtheorem{theorem}{Theorem}
\newtheorem{definition}{Definition}
\newtheorem{problem}{Problem}
\newtheorem{proposition}{Proposition}
\newtheorem{remark}{Remark}
\newtheorem{assumption}{Assumption}

\title{\LARGE \bf
	Geometric Reduced-Attitude Tracking under a Time-Varying Conic Constraint via Smooth Reference-Shaping*
}

\author{Pedro Santos$^{1}$, Joel Reis$^{2}$, Paulo Oliveira$^{1}$, and Carlos Silvestre$^{2,3}$
	\thanks{*This article has been accepted for publication in IEEE Control Systems Letters. This is the authors' version which has not been fully edited and content may change prior to final publication. Citation information: DOI 10.1109/LCSYS.2026.3707149. \copyright 2026 IEEE.  Personal use of this material is permitted.  Permission from IEEE must be obtained for all other uses, in any current or future media, including reprinting/republishing this material for advertising or promotional purposes, creating new collective works, for resale or redistribution to servers or lists, or reuse of any copyrighted component of this work in other works. }
\thanks{This work was supported %
	by Fundação para a Ciência e Tecnologia (FCT), through the Ph.D. scholarship of P. Santos (2023.00268.BD), LAETA (UID/50022/2025), and LARSyS (LA/P/0083/2020), %
	by the Macao Science and Technology Development Fund under Grants 0192/2023/RIA3, 0059/2024/RIA1, 0036/2025/AIJ and 0021/2025/ASJ, %
	and by the University of Macau, under Projects MYRG2022-00205-FST, MYRG-GRG2023-00107-FST-UMDF, UMDF-TISF/2025/007/FST, and SRG2024-00012-FST. %
}
\thanks{$^{1}$P. Santos and P. Oliveira are with the Institute of Mechanical Engineering and the Institute for Systems and Robotics, Instituto Superior Técnico, Universidade de Lisboa, Lisbon, Portugal
        {\tt\small \{pedrodossantos31, paulo.j.oliveira\}@tecnico.ulisboa.pt}}%
\thanks{$^{2}$J. Reis and C. Silvestre are with the Department of Electrical and Computer Engineering of the Faculty of Science and Technology, University of Macau, Taipa, Macao, China
        {\tt\small \{joelreis, csilvestre\}@um.edu.mo}}%
\thanks{$^{3}$C. Silvestre is on leave from the Institute for Systems and Robotics,  Instituto Superior Técnico, Universidade de Lisboa, Lisbon, Portugal.
       }%
}

\begin{document}

\maketitle
\thispagestyle{empty}
\pagestyle{empty}

\begin{abstract}

This letter studies reduced-attitude tracking for a rigid body on the 2-sphere $\mathbb{S}^2$ under a time-varying conic constraint. Using a kinematic model on $\mathbb{S}^2$, we first propose a geometric tracking law that guarantees almost-global asymptotic and regionally exponential convergence in the unconstrained case, where the angular velocity serves as the control input. We then introduce a smooth reference-shaping mechanism that adjusts the desired direction so that the reference provided to the controller satisfies the time-varying conic constraint while preserving the smoothness required by the tracking law. The resulting approach yields smooth continuous feedback and retains the stability guarantees of the unconstrained controller, albeit at the expense of enforcing a soft version of the original constraint. Simulation results illustrate the effectiveness of the method and highlight its suitability for applications where deterministic behavior, smooth control action, and strong stability guarantees are preferred over hard constraint satisfaction.

\end{abstract}


\section{INTRODUCTION}




Reduced-attitude control of rigid bodies arises in a wide range of applications, motivated by the need to regulate a body-fixed axis and by the presence of torque underactuation. Representative examples include spacecraft spin-axis stabilization \cite{Tsiotras1994}, thrust-direction control for trajectory tracking of underactuated vehicles \cite{Casau19,Leomanni2024}, and visual servoing \cite{Serra2016}. While the full attitude of a rigid body is naturally represented by a rotation matrix in $\mathrm{SO}(3)$, the reduced attitude is more appropriately represented by a unit vector on $\mathbb{S}^2$. By exploiting the geometric properties of this Riemannian manifold, one can design singularity-free control laws and obtain almost-global stabilization and tracking results \cite{Bullo95}.

Motivated by real engineering problems, recent years have seen a growing interest in the constrained attitude control problem \cite{Fiori24}, in which specific regions of the configuration space, either SO(3) or $\mathbb{S}^2$, must be avoided \cite{Lee17}, \cite{Liu23}. The literature on constrained attitude control is vast and can be broadly classified according to control task, constraint type, and solution strategy for handling constraints. The control task may involve reorientation, stabilization or tracking. Constraints may be single \cite{Ibuki2020} or multiple \cite{Dimos21}, and can be static or dynamic \cite{Kim2004}, of the inclusion or exclusion type \cite{Hu2019}, and may have different geometric structures \cite{Sawant25}. Finally, the most predominant solution strategies include reference-shaping approaches \cite{Nicotra20,Nakano2018,Tan20}; barrier-based methods, such as barrier Lyapunov functions \cite{Romdlony16}, potential functions \cite{Fiori24}, \cite{Mesbahi14}, and Control Barrier Functions (CBFs) \cite{Ibuki2020}, \cite{Wu15}; and optimization-based methods, including convex optimization \cite{Szmuk20} and Model Predictive Control (MPC) \cite{Guiggiani2015}.

Relatively few works address attitude trajectory tracking under time-varying constraints. In particular, solutions for conic constraints in which both the cone axis and the half-angle vary with time remain limited. In \cite{Hu2019}, two dynamic conic constraints under reduced-attitude tracking are handled via artificial potential and Lyapunov barrier functions, although the half-angle limits are kept constant. Other works have focused on the reorientation/planning problem with time-varying constraints. For instance, in \cite{He2022} the authors rely on a search-based method to determine the attitude maneuver between two points, while in \cite{Li2025} an artificial potential field-based method is proposed for the reorientation problem. Moreover, both single and multiple invariant conic constraints have been studied in \cite{Ibuki2020}, \cite{Ramos18}, \cite{Shen18}. 

In this work, we consider a single dynamic conic constraint on $\mathbb{S}^2$, where both the cone axis and half-angle are known functions of time. This scenario is relevant to applications including angle-of-attack limiting in aerial vehicles \cite{Auntenrieb25}, where the relative velocity vector is time-varying and the maximum allowable angle depends on flight conditions, as well as spacecraft solar panel and sensor pointing \cite{Mesbahi14}, where pointing constraints vary with orbital dynamics. By leveraging a reduced-attitude kinematic model on $\mathbb{S}^2$ and treating two orthogonal components of the angular velocity vector as control inputs, we first devise a geometric reduced-attitude tracking law for the unconstrained case. Related approaches to the unconstrained reduced-attitude tracking problem can be found in \cite{Lee2016,Pong2015,Coates2021}. We then introduce a smooth reference‑shaping mechanism that generates an admissible and sufficiently smooth reference by leveraging the notion of geodesic distance and employing a smooth ramp function. By feeding the controller with the admissible shaped-reference, the closed-loop system is able to satisfy a soft version of the conic constraint and the tracking error is minimized up to a sub-optimal region around the constraint boundary introduced by a smoothing step.

\subsection{Contributions}

Aiming at strict stability guarantees through smooth feedback, the main contributions can be summarized as follows:
\begin{itemize}
	\item  We propose a smooth reference-shaping mechanism for a time-varying conic constraint, ensuring admissibility of the shaped-reference and preserving the differentiability properties required by the controller;
	\item We prove that the tracking controller, fed with the shaped-reference, ensures satisfaction of a soft version of the constraint, providing a closed-form solution;
	\item We quantify the deviation from the optimal projected reference introduced by the smoothing step.
\end{itemize}

\subsection{Notation}

Bold lowercase and uppercase symbols stand for column vectors and matrices, respectively. The $n$-dimensional Euclidean space is represented by ${\mathbb{R}}^n$. Given a vector $\b{x} \in \mathbb{R}^n$, its Euclidean norm is defined as $||\b{x}|| \coloneq \sqrt{\b{x}^\mathsf{T}\bf{x}}$. The 2-sphere is defined as $\mathbb{S}^2\coloneq\left\{\b{x}\in\mathbb{R}^3:\b{x}^\mathsf{T}\b{x}=1\right\}$. The space tangent to $\mathbb{S}^2$ at a point $\b{p}\in\mathbb{S}^2$ is denoted by $T_\b{p}\mathbb{S}^2$. The special orthogonal group of order three is defined as $\mathrm{SO}(3)$$\coloneq\left\{\b{X}\in \mathbb{R}^{3\times3}: \b{X}\b{X}^\mathsf{T}=\b{X}^\mathsf{T} \b{X} = \b{I},\,\text{det}(\b{X})=1\right\}$. The boundary of a set $\mathcal{A}$ is denoted by $\partial \mathcal{A}$. The operator $\b{S}(\cdot):\mathbb{R}^3\mapsto \mathbb{R}^{3\times3}$ yields a skew-symmetric matrix satisfying (the cross product) $\b{S}(\b{x})\b{y} =\b{x} \times \b{y} $, for any $\b{x},\,\b{y}\in \mathbb{R}^3$. The symbol $\mathbf{I}$ denotes the identity matrix of appropriate dimensions. A function $f$ is of class $\mathcal{C}^n$ if its derivatives $f^{(1)}$, $f^{(2)}$, ...,$f^{(n)}$ exist and are continuous.   The vectors $\b{e}_1,\,\b{e}_2,\,\b{e}_3 \in \mathbb{S}^2$ are orthonormal and form a basis of $\mathbb{R}^3$, such that ${\b{I}}=\left[{\b{e}}_1\;\b{e}_2\;\b{e}_3\right]$. 

\section{Problem Formulation}
We consider a rigid body and address the problem of generating an angular velocity command to track a time-varying inertial pointing reference in the presence of a time-varying conic constraint. Given a body-fixed frame $\{B\}$, typically aligned with the principal directions of the body and attached to its center of mass, and an inertial frame $\{I\}$, the configuration manifold of the attitude kinematics is the special orthogonal group $\mathrm{SO}(3)$. By defining $\b{R}\in\mathrm{SO}(3)$ as the rotation matrix that transforms vectors from $\{B\}$ to $\{I\}$, the attitude kinematics are given by $\dot{\b{R}} = \b{R}\,\b{S}(\bs{\omega})$, where $\bs\omega \in \mathbb{R}^3$ is the angular velocity of $\{B\}$ with respect to $\{I\}$, expressed in $\{B\}$.

\subsection{Reduced-Attitude Kinematics on the $2$-Sphere }
When controlling a single body-fixed direction, the problem is more naturally posed as one of reduced‑attitude tracking, where rotation about this axis is left unattended. To that end, we make use of the projection map \cite{Bullo95}
\begin{equation}\label{eq:projector}
	\bs{\pi}_i(\b{R}) \coloneq \b{R}\b{e}_i\,,\hspace{10pt}i=\left\{1,2,3\right\}
\end{equation}
which takes the full attitude of the vehicle and projects it into a unit vector on $\mathbb{S}^2$. Using this map, the axes of $\{B\}$ expressed in $\{I\}$ are denoted by $\b{q}_i \in \mathbb{S}^2$, and given by
\begin{equation}\label{eq:projector2}
	\b{q}_i \coloneq  \bs{\pi}_i(\b{R})\,,\hspace{10pt}i = \{1,\,2,\,3\}\,.
\end{equation}
From \eqref{eq:projector} and \eqref{eq:projector2}, the reduced-attitude kinematics for each axis $\b{q}_i$ can be obtained and expressed as follows:
\begin{equation}\label{eq:kinematics}
	\dot{\b{q}}_i = -\b{S}(\b{q}_i)\b{u}_i\,, \hspace{10pt} 	\dot{\b{q}}_i\in T_{\b{q}_i}\mathbb{S}^2\,,
\end{equation}
where $\b{u}_i\in T_{\b{q}_i}\mathbb{S}^2$ is defined as $\b{u}_i \coloneq -\b{S}^2(\b{q}_i)\b{R}\,\bs\omega\,$, corresponding to the projection of $\b{R}\,\bs{\omega}$ onto $T_{\b{q}_i}\mathbb{S}^2$. By noting that $\b{R}\,\bs\omega=\b{e}_1^\mathsf{T}\bs\omega\b{q}_1 + \b{e}_2^\mathsf{T}\bs\omega\b{q}_2 + \b{e}_3^\mathsf{T}\bs\omega\b{q}_3$, it follows, for distinct $i,j,k \in \{1,2,3\}$, that
\begin{equation}\label{eq:input}
	\b{u}_i = \b{e}_j^\mathsf{T}\bs\omega\b{q}_j + \b{e}_k^\mathsf{T}\bs\omega\b{q}_k\,.
\end{equation}
In other words, the dynamics of each axis $\b{q}_i$ are only impacted by the components of $\bs\omega$ orthogonal to it.
\begin{assumption}\label{ass:u}
	The angular velocity $\bs\omega$ is regulated by a sufficiently fast inner-loop controller. Consequently, $\b{u}_i$ in \eqref{eq:input} can be treated as a control input.
\end{assumption}
From this point onward, without loss of generality, we drop the index notation to consider a generic body axis $\b{q}\in\mathbb{S}^2$ and, following Assumption~\ref{ass:u}, an associated control input $\b{u}\in T_{\b{q}}\mathbb{S}^2$. Then, the unconstrained reduced-attitude tracking problem can be stated as follows:
\begin{problem}[\textbf{Unconstrained reference tracking}]\label{prob:unconstrained}
	Given a time-parameterized desired direction denoted by $\b{q}_d\in\mathbb{S}^2$, design a kinematic steering law for the input $\b{u}$ such that, under the model in \eqref{eq:kinematics}, $\b{q}(t)\rightarrow\b{q}_d(t)$ as $t\rightarrow\infty$.
\end{problem}
\begin{assumption}\label{ass:class}
	The reference $\b{q}_d(t)$ is of class at least $\mathcal{C}^1$.
\end{assumption}

\subsection{Time-Varying Pointing Constraint}
\begin{figure}[t]
	\centering
	\includegraphics[width=0.79\columnwidth]{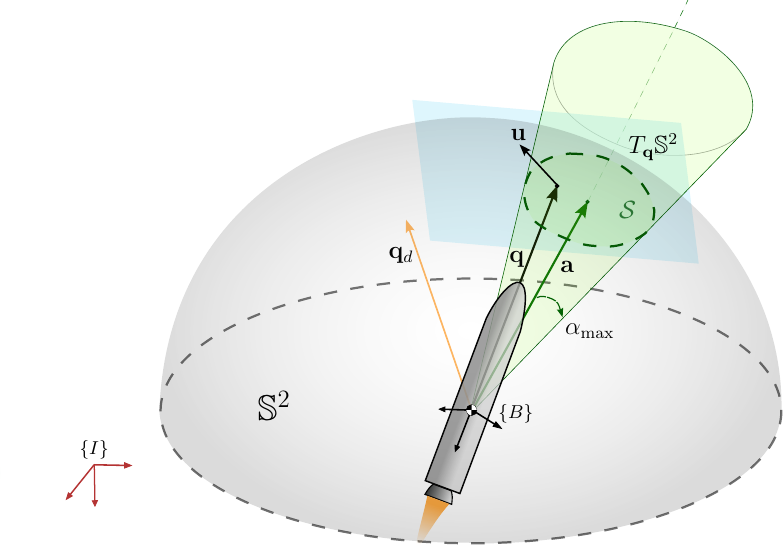} 
	\caption{Problem formulation on the 2-sphere for an inclusion constraint.}
	\label{fig:sphere}
\end{figure}

Besides the reference tracking objective, we wish to constrain the geodesic distance between $\b{q}$ and an externally defined, time-parameterized direction $\b{a}\in\mathbb{S}^2$.

\begin{definition}[\textbf{Geodesic distance} \cite{Bullo95}]
	The geodesic distance between two points $\b{x},\,\b{y}\in\mathbb{S}^2$ is defined as the angle between them and is given by $d_{\mathbb{S}^2}(\b{x},\b{y}) \coloneq \arccos{(\b{x}^\mathsf{T}\b{y})}$, with $0 \leq d_{\mathbb{S}^2}({\bf x}, {\bf y}) \leq \pi$, which corresponds to the length of the shorter great-circle arc connecting the points.
\end{definition}
The time-varying constraint can be formulated as 
\begin{equation}\label{eq:conic_constraint}
	d_{\mathbb{S}^2}(\b{q}(t),\b{a}(t)) \leq \alpha_\mathrm{max}(t)\,, 
\end{equation}
where $\alpha_\text{max} \in \left(0,\,\pi\right)$ is an externally defined, possibly time-varying, maximum value for the geodesic distance between $\b{q}$ and $\b{a}$. Accordingly, a safe set $\mathcal{S}$ is defined as $\mathcal{S}(t)\coloneq\left\{\b{q}\in\mathbb{S}^2:d_{\mathbb{S}^2}(\b{q}(t),\b{a}(t))\leq\alpha_\mathrm{max}(t)\right\}$. Note that the formulation in \eqref{eq:conic_constraint} can be used to represent both conic inclusion and exclusion constraints. Fig.~\ref{fig:sphere} schematizes the constrained problem through an example scenario in which reduced-attitude tracking of a launch vehicle must satisfy an angle-of-attack conic inclusion constraint. 

In this letter, we consider a relaxed version of the conically constrained attitude control problem, where the hard constraint \eqref{eq:conic_constraint} is replaced by the soft constraint
\begin{equation}\label{eq:soft_conic_constraint}
	d_{\mathbb{S}^2}(\b{q}(t),\b{a}(t)) \leq \alpha_\mathrm{max}(t) + \gamma(t)\,,
\end{equation}
where $\gamma:[0,\infty) \mapsto  \left(0,\pi\right]$ is an admissible vanishing violation, satisfying $\gamma(t)\to0$ as $t\to\infty$ and $\dot{\gamma}(t)<0$ for $t\geq0$. The soft-constrained problem can then be stated as follows: 
\begin{problem}[\textbf{Soft-constrained reference tracking}]\label{prob:constrained}
	Given a time-parameterized desired direction $\b{q}_d\in\mathbb{S}^2$, design a control law for $\b{u}$ such that, under the model in \eqref{eq:kinematics}, $d_{\mathbb{S}^2}(\b{q}(t),\b{q}_d(t))$ is minimized while satisfying constraint \eqref{eq:soft_conic_constraint}.  
\end{problem}
\begin{assumption}\label{ass:smooth_constraint}
	The vector $\b{a}$ and the angle $\alpha_\mathrm{max}$ are given by time-parameterized functions that are at least of the same differentiability class as the reference $\b{q}_d$ and $\b{q}_d(t) \neq -\b{a}(t)$ for all $t\geq0$. 
\end{assumption}
Due to the topological properties of $\mathbb{S}^2$, Assumption~\ref{ass:smooth_constraint} will prove to be useful when solving Problem~\ref{prob:constrained}.

\subsection{Solution Strategy}

Our approach aims to enforce stability guarantees through continuous smooth feedback while avoiding online optimization. To that end, we start by solving the unconstrained Problem~\ref{prob:unconstrained} and design a geometric controller that ensures almost global convergence of $\b{q}$ to $\b{q}_d$. Then, we address Problem~\ref{prob:constrained} by proposing a smooth reference-shaping method such that the condition
\begin{equation}\label{eq:relaxed_constraint}
	d_{\mathbb{S}^2}(\overbar{\b{q}}_d(t),\b{a}(t)) \leq \alpha_\mathrm{max}(t)\,, 
\end{equation}
where $\overbar{\b{q}}_d\in\mathbb{S}^2$ is the shaped-reference supplied to the controller, is always satisfied. The resulting scheme is shown in Fig. \ref{fig:diagram}. As opposed to optimization- and/or CBF-based methods, our strategy foregoes a strict certificate of forward-invariance of the safe set $\mathcal{S}$ in the name of improved real-time stability guarantees and smoother control input.

\section{Unconstrained Tracking Controller}
Consider the Lyapunov function $V:\mathbb{S}^2\times\mathbb{S}^2\mapsto\left[0,2\right]$ 
\begin{equation}\label{eq:V}
	V(\b{q},\b{q}_d) \coloneq 1-\b{q}^\mathsf{T}_d\b{q}\,.
\end{equation}
Using \eqref{eq:kinematics}, its time derivate can be written as
\begin{equation}\label{eq:Vdot1}
	\dot{V} =-\b{q}^\mathsf{T}\left(\b{S}(\b{q}_d)\b{u}+\dot{\b{q}}_d \right) = \left(\b{S}(\b{q}_d)\b{q}\right)^\mathsf{T}\left( \b{u}-\b{S}(\b{q}_d)\dot{\b{q}}_d\right)\,,
\end{equation}
where the fact that $\dot{\b{q}}_d\in  T_{\b{q}_d}\mathbb{S}^2 $ was used in the second step. By defining a tracking error $\b{e}\in T_\b{q}\mathbb{S}^2\cap T_{\b{q}_d}\mathbb{S}^2 $ as
\begin{equation}\label{eq:e}
	\b{e}\coloneq \b{S}(\b{q}_d)\b{q}\,,
\end{equation}
and defining a desired angular velocity $\bs{\Omega}_d\in T_{\b{q}_d}\mathbb{S}^2$ as $\bs{\Omega}_d\coloneq\b{S}(\b{q}_d)\dot{\b{q}}_d$, the time derivative in \eqref{eq:Vdot1} simplifies to $\dot{V} = \b{e}^\mathsf{T}\left(\b{u} - \bs{\Omega}_d\right)$. Therefore, by setting 
\begin{equation}\label{eq:unconstrainedlaw}
	\b{u}\coloneq - k\b{e} -\b{S}^2(\b{q})\,\bs\Omega_d\,,
\end{equation}
where $k>0$ is a user-defined gain, the time derivative of \eqref{eq:V} becomes
\begin{equation}\label{eq:Vdot3}
	\dot{V} = - k\,||\b{e}||^2\leq 0\,,
\end{equation}
where the fact that $\b{e}^\mathsf{T}\left(\b{S}^2(\b{q}) + \b{I}\right)\bs\Omega_d = 0$ was employed.
\begin{proposition}\label{prop:tracking}
	Under Assumptions~\ref{ass:u} and \ref{ass:class}, the control law \eqref{eq:unconstrainedlaw} renders the origin of the tracking error $\b{e}$ in \eqref{eq:e} almost globally asymptotically stable with a region of exponential stability given by $\mathcal{Q} \coloneq \{\b{q} \in \mathbb{S}^2: d_{\mathbb{S}^2}(\b{q},\b{q}_d) \leq\pi/2\}$.
\end{proposition}
\begin{proof}
	The limits of the codomain of $V$ occur at the points where $\b{e}$ is zero, i.e., $V(\b{q}=\b{q}_{d}) = 0$ and $V(\b{q}=-\b{q}_{d})=2$. Moreover, its time derivative in (\ref{eq:Vdot3}) is negative definite with respect to $\b{e}$, vanishing when $\b{q}=\pm\b{q}_{d}$. Therefore, for any initial condition in the set $\mathcal{I} \coloneq \{\b{q} \in \mathbb{S}^2: \b{q}\neq-\b{q}_{d}\}$, $\b{q}$ will converge asymptotically to $\b{q}_{d}$. Moreover, in the set $\mathcal{I}$, the relation $||\b{e}||^2 = V\left(2-V\right)$ is valid, where we used \eqref{eq:V} and \eqref{eq:e}, allowing us to rewrite \eqref{eq:Vdot3} as $\dot{V} = -2kV + kV^2$. In the set $\mathcal{Q} \coloneq \{\b{q} \in \mathbb{S}^2: d_{\mathbb{S}^2}(\b{q},\b{q}_d) \leq\pi/2\}$, we have that $0 \leq V \leq 1$, allowing us to upper bound the time derivative of $V$ according to $\dot{V} \leq -2kV + kV = -kV$. Thus, when $\b{q}$ enters the set  $\mathcal{Q}$ at $t=t_i$, for $t_i\geq0,$ we have that
	\begin{equation}\label{eq:exponential_convergence}
		V(t) \leq V(t_i)\, e^{-kt}\,,\hspace{10pt} \forall t \geq t_i\,,
	\end{equation}
	implying that $\b{q}$ converges exponentially to $\b{q}_d$.
\end{proof}
\begin{remark}
	By replacing $\b{e}$ in the control law \eqref{eq:unconstrainedlaw} with $\b{e}/(1+\b{q}_d^\mathsf{T}\b{q})$, the origin of $\b{e}$ becomes almost globally exponentially stable, at the expense of the control law becoming undefined at $\b{q}=-\b{q}_d$.
\end{remark}
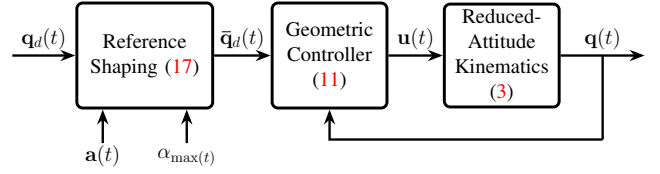
\begin{figure}[t]
	\centering
	\resizebox{\columnwidth}{!}{%
		\begin{tikzpicture}[>=Stealth, node distance=1cm and 1.5cm]

\node[
	draw=black,
	line width=2pt,
	rounded corners,
	minimum height=2.5cm,
	text width=3.0cm,
	align=center,
	font=\Large,
] 
(ReferenceShaping) at (-5,-1) 
{
	Reference Shaping \eqref{eq:smooth_ref}
};

\draw[black,->,ultra thick] 
	([xshift=-1.5cm]ReferenceShaping.west) -- (ReferenceShaping.west)
	node[midway, above, font=\Large] {$\mathbf{q}_d(t)$};
	
\node[
	draw=black,
	line width=2pt,
	rounded corners, 
	minimum height=2.5cm, 
	text width=2.5cm,
	align=center,
	font=\Large,
] 
(GeometricController) at (-0.6,-1) 
{
	Geometric Controller \eqref{eq:unconstrainedlaw}
};
	
\draw[black,->,ultra thick] 
	(ReferenceShaping.east) 
	to node[midway, above, font=\Large] {$\overbar{\mathbf{q}}_d(t)$}
	(GeometricController.west);
	
\node[
	draw=black,
	line width=2pt,
	rounded corners, 
	minimum height=2cm,
	text width=2.5cm,
	align=center,
	font=\Large,
] 
(RAK) at (3.5,-1) 
{
	Reduced-Attitude Kinematics \eqref{eq:kinematics}
};

\draw[black,->,ultra thick] 
	(GeometricController.east) 
	to node[midway, above, font=\Large] {$\mathbf{u}(t)$}
	(RAK.west);

\node[] (Output) at (7,-3.5) {};

\draw[black,->,ultra thick]
	(RAK.east) -- node[above,font=\Large,xshift=0cm] {$\mathbf{q}(t)$}
	([xshift=2cm]RAK.east);
	
\draw[black,-,ultra thick]
	([xshift=1cm]RAK.east)
	-|
	([xshift=1cm,yshift=-2.0cm]RAK.east);

\draw[black,->,ultra thick] ([xshift=-1cm,yshift=-0.8cm]ReferenceShaping.south) -- node[below,font=\Large,yshift=-0.3cm] {$\mathbf{a}(t)$} ([xshift=-1cm]ReferenceShaping.south);

\draw[black,->,ultra thick] ([xshift=1cm,yshift=-0.8cm]ReferenceShaping.south) -- node[below,font=\Large,yshift=-0.4cm] {$\alpha_{\max(t)}$} ([xshift=1cm]ReferenceShaping.south);

\draw[black,->,ultra thick] ([xshift=1cm,yshift=-2.0cm]RAK.east) -| (GeometricController.south);

\end{tikzpicture}
	}
	\caption{Diagram of the proposed solution.}
	\label{fig:diagram}
\end{figure}

\section{Smooth Reference-Shaping}
With a kinematic steering law \eqref{eq:unconstrainedlaw} that solves the unconstrained Problem~\ref{prob:unconstrained} already available, we proceed by proposing a smooth reference-shaping mechanism such that the reference supplied  to the controller, $\overbar{\b{q}}_d$, does not violate the constraint in \eqref{eq:relaxed_constraint}. To that end, we start by introducing the operator $\bs\Lambda(\chi,\b{x}):[0,\pi]\times\mathbb{S}^2 \mapsto \mathrm{SO}(3)$ which performs a rotation about the vector $\b{x}\in\mathbb{S}^2$ by an angle $\chi$. According to Rodrigues' formula, we have
\begin{equation}\label{eq:Rodrigues}
	\bs\Lambda(\chi,\b{x})  \coloneq \b{I} + \sin\chi\,\b{S}(\b{x}) + (1-\cos\chi)\,\b{S}^2(\b{x}) \,.
\end{equation}
\subsection{Optimal Reference-Shaping}
The optimal shaped-reference, denoted by $\overbar{\b{q}}_d^\star \in\mathbb{S}^2$, is defined as the element that minimizes $d_{\mathbb{S}^2}(\b{q}_d,\overbar{\b{q}}_d)$ while satisfying constraint \eqref{eq:relaxed_constraint}, and is given by
\begin{subequations}\label{eq:opt_shape}
	\begin{equation}
		\overbar{\b{q}}_d^\star(t) \coloneq \begin{cases}
			\b{q}_d(t)\,,\hspace{20pt}d_{\mathbb{S}^2}(\b{q}_d,\b{a})\leq\alpha_\mathrm{max}\\
			\b{q}_{d}'(t)\,,\hspace{20pt}d_{\mathbb{S}^2}(\b{q}_d,\b{a})>\alpha_\mathrm{max}
		\end{cases}\,,
	\end{equation}
	where $\b{q}_{d}'\in\mathbb{S}^2$ is the vector on $\partial\mathcal{S}$ with minimal geodesic distance to $\b{q}_d$. Using \eqref{eq:Rodrigues}, it follows that 
	\begin{equation}\label{eq:Lambda}
		\b{q}_{d}'(t)\coloneq\bs\Lambda(\theta(t),\b{v}(t))\,\b{q}_d(t)\,,
	\end{equation}
	\begin{equation}\label{eq:theta_def}
		\theta(t) \coloneq d_{\mathbb{S}^2}(\b{q}_d,\b{q}_d') = d_{\mathbb{S}^2}(\b{q}_d,\b{a}) - \alpha_\mathrm{max}\,,
	\end{equation}
	\begin{equation}
		\b{v}(t) \coloneq \frac{\b{S}(\b{q}_d)\b{a}}{||\b{S}(\b{q}_d)\b{a}||}\,.
	\end{equation}
\end{subequations}
Note that, under Assumption~\ref{ass:smooth_constraint}, $\overbar{\b{q}}_d^\star$ is defined everywhere. The optimal shaped-reference $\overbar{\b{q}}_d^\star$ in \eqref{eq:opt_shape} is a saturated version of $\b{q}_d$, limited by the conic constraint, which minimizes $d_{\mathbb{S}^2}(\b{q}_d,\overbar{\b{q}}_d)$ whenever $\b{q}_d$ is outside the safe region $\mathcal{S}$. 
\subsection{Smooth Reference-Shaping}
The hard saturation behavior of $\overbar{\b{q}}_d^\star$ in \eqref{eq:opt_shape} violates the smoothness requirements of the tracking controller. To tackle this, a smooth version of the optimal shaped-reference is derived by addressing the non-smooth transition at the constraint boundary, when $d_{\mathbb{S}^2}(\b{q}_d,\b{a})=\alpha_\text{max}$. Let $\psi(\eta):[0,1]\mapsto[0,1]$ be a class $\mathcal{C}^n$ function satisfying the following properties: $\psi(0)=0$, $\psi(1)=1$, $\psi^{(i)}(0)=\psi^{(i)}(1)=0$ for $1\leq i\leq n$, $\psi(\eta)+\psi(1-\eta)=1$, and $\psi^{(1)}\geq0$. Using $\psi$, a smooth approximation function to the ramp function $H(x)x$, where $H(x)$ is the Heaviside step, can be constructed as
\begin{equation}\label{eq:smooth_transition}
	w (x) \coloneq \begin{cases}
		0\,,\hspace{83pt}x\leq-\delta\\
		\int_{-\delta}^{x}\psi(\frac{\eta+\delta}{2\delta})\,d\eta\,,\hspace{20pt}-\delta<x<\delta\\
		x\,,\hspace{83pt}x\geq\delta		
	\end{cases}\,,
\end{equation}
where $w(x):\mathbb{R}\mapsto[0,\infty)$ is of class $\mathcal{C}^{n+1}$ and $\delta\in\left(0,1\right)$. Then, the shaped-reference is proposed as
\begin{subequations}\label{eq:smooth_ref}
	\begin{equation}\label{eq:smooth_ref_2}
		\overbar{\b{q}}_d(t) \coloneq \begin{cases}
			\b{q}_d(t)\,,\hspace{46.5pt}0\leq d_{\mathbb{S}^2}(\b{q}_d,\b{a}) < \alpha_\mathrm{max}(1-\delta)\\[10pt]
			\bs\Lambda\left(\overbar{\theta}(t),\b{v}(t)\right)\b{q}_d(t)\,,\hspace{6pt}d_{\mathbb{S}^2}(\b{q}_d,\b{a})\geq\alpha_\mathrm{max}(1-\delta)        
		\end{cases}
	\end{equation}
	\begin{equation}\label{eq:theta_bar_def}
		\overbar{\theta}(t) \coloneq w\left(\frac{\theta}{\alpha_\mathrm{max}}\right)\,\alpha_\mathrm{max}\,.
	\end{equation}
\end{subequations}
Looking at \eqref{eq:smooth_ref}, and comparing it with \eqref{eq:opt_shape}, it is clear that $\overbar{\b{q}}_d$ is equivalent to its optimal counterpart, $\overbar{\b{q}}_d^\star$, whenever $\b{q}_d$ lies outside an arbitrarily small region around the constraint given by $d_{\mathbb{S}^2}(\b{q}_d,\b{q}_d') \leq \delta\alpha_\mathrm{max}$. This sub-optimal region is introduced with the goal of smoothing out the hard saturation behavior of the optimal solution at the boundary transition. The construction of $\overbar{\theta}$ in \eqref{eq:theta_bar_def} is such that it smoothly matches the original rotation angle $\theta$ outside the transition region. Note that for $\theta/\alpha_\mathrm{max}\geq\delta$ we have that $\overbar{\theta}=\theta$.
\begin{remark}\label{remark:class_w}
	There exist infinite functions $\psi(\eta)$ for \eqref{eq:smooth_transition} which satisfy the stated requirements for an arbitrary $n$. In particular, the family of smoothstep polynomials \cite{Ebert2002}, obtained via Hermite interpolation of the boundary conditions, and given by
	\begin{equation}\label{eq:psi}
		\psi(\eta) = \eta^{n+1} \displaystyle\sum_{j=0}^{n} \frac{(2n+1)!}{j!\,n!(n-j)!(n+j+1)}(-\eta)^j\,,
	\end{equation}
	yields suitable candidates.
\end{remark}                                                                          
\begin{proposition}\label{prop:smooth} Suppose that $\b{q}_d$ is of class $\mathcal{C}^k$. Under Assumption~\ref{ass:smooth_constraint}, for $n\geq k-1$, the reference shaping law in \eqref{eq:smooth_ref} is well-defined and $\overbar{\b{q}}_d$ is also of class $\mathcal{C}^k$.
\end{proposition}
\begin{proof}
	Given that only smoothness preserving operations and mappings are used, each branch in \eqref{eq:smooth_ref_2} is of class $\mathcal{C}^k$ in its respective domain. Note that $w(x)$ is of class $\mathcal{C}^{n+1}$. Hence, we have left to prove the continuity of $\overbar{\b{q}}_d$ and of its $k$-th order derivatives at the transition point defined by $d_{\mathbb{S}^2}(\b{q}_d,\b{a}) = \alpha_\text{max}(1-\delta)$. From \eqref{eq:theta_def}, that point corresponds to $\theta/\alpha_\mathrm{max}=-\delta$, leading to $\overbar{\theta}=0$, which in turn implies ${\bs\Lambda}\left(\overbar{\theta},\b{v}\right) = \b{I}$, thereby proving continuity. At the point $x=-\delta$, from the definition of $\psi$, it follows that $w^{(i)}(x) =0$ for $i\leq k\leq n-1$. With this, and based on \eqref{eq:Rodrigues}, one can infer that $\frac{d^{i}}{dt^{i}}\left({\bs\Lambda\left(\overbar{\theta}(t),\b{v}(t)\right)}\b{q}_d(t)\right)|_{\overbar{\theta}=0}  = 	\frac{d^{i}}{dt^{i}}\b{q}_d(t)$ for any $i\leq k$, thus concluding our proof.
\end{proof}

We can now state the main theorem regarding the solution of Problem~\ref{prob:constrained} through the interconnection of the unconstrained tracking law with the smooth reference-shaping.
\begin{theorem}\label{th}
	Consider the sets $\mathcal{I} \coloneq \{\b{q} \in \mathbb{S}^2: \b{q}\neq-\overbar{\b{q}}_{d}\}$ and $\mathcal{Q}\coloneq \{\b{q} \in \mathbb{S}^2: d_{\mathbb{S}^2}(\b{q},\overbar{\b{q}}_d) \leq\pi/2\}$. Under Assumptions~\ref{ass:u}, \ref{ass:class}, and \ref{ass:smooth_constraint}, for $\b{q}(0)\in\mathcal{I}$, the control law in \eqref{eq:unconstrainedlaw} together with the smooth reference-shaping mechanism in \eqref{eq:smooth_ref} ensures that constraint \eqref{eq:soft_conic_constraint} is satisfied with $\gamma(t) = 	d_{\mathbb{S}^2}(\b{q},\overbar{\b{q}}_d)$ and  $d_{\mathbb{S}^2}(\b{q},\b{q}_d)$ is minimized in a $\delta\gamma$-optimal sense, i.e.,
	\begin{equation*}
		|d^\star_{\mathbb{S}^2}(\b{q},\b{q}_d)-d_{\mathbb{S}^2}(\b{q},\b{q}_d)|\leq 2\delta\,\alpha_\mathrm{max}\,\int_{0}^{1/2}\psi(\eta)\,d\eta + \gamma\,,
	\end{equation*}
	\begin{equation}
		d^\star_{\mathbb{S}^2}(\b{q},\b{q}_d)\coloneq \begin{cases}
			0\,,\hspace{47pt}d_{\mathbb{S}^2}(\b{q}_d,\b{a})\leq\alpha_\mathrm{max}\\
			d_{\mathbb{S}^2}(\overbar{\b{q}}^\star_d,\b{q}_d)	\,,\hspace{5pt}d_{\mathbb{S}^2}(\b{q}_d,\b{a})>\alpha_\mathrm{max}
		\end{cases}.
	\end{equation}
	Moreover, for $\b{q}(0)\in\mathcal{Q}$, the time instant $t_s$ after which $\gamma(t)$ is below a given geodesic distance $\gamma_s>0$ is bounded by
	\begin{equation}\label{eq:ts}
		t_s \leq \frac{1}{k}\ln\left(\frac{1-\b{q}^\mathsf{T}(0)\overbar{\b{q}}_d(0)}{1-\cos\gamma_s}\right)\,.
	\end{equation}
\end{theorem}

\begin{proof}
	From the triangle inequality on $\mathbb{S}^2$, we have that  $d_{\mathbb{S}^2}(\b{q},\b{a})\leq d_{\mathbb{S}^2}(\b{q},\overbar{\b{q}}_d) +  d_{\mathbb{S}^2}(\overbar{\b{q}}_d,\b{a})$. By construction, $d_{\mathbb{S}^2}(\overbar{\b{q}}_d,\b{a})\leq\alpha_\mathrm{max}$, which implies that $d_{\mathbb{S}^2}(\b{q},\b{a})\leq \alpha_\mathrm{max} + d_{\mathbb{S}^2}(\b{q},\overbar{\b{q}}_d)$. From Proposition~\ref{prop:smooth}, $\overbar{\b{q}}_d$ preserves the smoothness of $\b{q}_d$, therefore it can be used in the control law in \eqref{eq:unconstrainedlaw}, with all formal guarantees of Proposition~\ref{prop:tracking} extending to $\overbar{\b{q}}_d$. In particular, $\b{q}\to\overbar{\b{q}}_d$ and $\frac{d}{dt}d_{\mathbb{S}^2}(\b{q},\overbar{\b{q}}_d)<0$, which proves the satisfaction of the soft constraint \eqref{eq:soft_conic_constraint} with $\gamma(t) = 	d_{\mathbb{S}^2}(\b{q},\overbar{\b{q}}_d)$. Additionally, the quantity $d_{\mathbb{S}^2}(\overbar{\b{q}}_d,\overbar{\b{q}}_d^\star)$ is non-zero on the interval $\theta/\alpha_\mathrm{max} \in \left[-\delta,\,\delta\right]$, with its maximum occurring at $\theta=0$ given the conditions imposed on $\psi$. Therefore, $d_{\mathbb{S}^2}(\overbar{\b{q}}_d,\overbar{\b{q}}_d^\star)$ is bounded according to
	\begin{equation}
		d_{\mathbb{S}^2}(\overbar{\b{q}}_d,\overbar{\b{q}}_d^\star) \leq 2\delta\,\alpha_\mathrm{max}\,\int_{0}^{1/2}\psi(\eta)\,d\eta\,.
	\end{equation}
	Hence, $|d^\star_{\mathbb{S}^2}(\b{q},\b{q}_d)-d_{\mathbb{S}^2}(\b{q},\b{q}_d)|\leq 2\delta\,\alpha_\mathrm{max}\,\int_{0}^{1/2}\psi(\eta)\,d\eta + \gamma$. Finally, for $\b{q}(0)\in\mathcal{Q}$, we have that $\b{q}$ converges exponentially to $\overbar{\b{q}}_d$. From the exponential convergence condition in \eqref{eq:exponential_convergence}, together with the Lyapunov function definition in \eqref{eq:V}, the result in \eqref{eq:ts} immediately follows.
\end{proof}
With Theorem~\ref{th}, we prove that, by relaxing the safety requirements, the closed-loop system satisfies the soft constraint \eqref{eq:soft_conic_constraint} while retaining almost global stability via smooth feedback. Moreover, it can recover from an initial condition outside the safe region. The additional tracking error introduced by the smooth transition depends on its relative size, which can be adjusted through the tuning parameter $\delta$. While reducing $\delta$ effectively approximates $\overbar{\b{q}}_d$ to $\overbar{\b{q}}_d^\star$, it also reduces its smoothness, leading to more aggressive changes in $\b{u}$ and potentially violating the limits of Assumption~\ref{ass:u}.

\section{Simulation}

In this section, computer simulation results are displayed for both static and dynamic conic inclusion constraints. The same reference trajectory is used for both scenarios and is given by $\b{q}_d(t)\coloneq\left[\,\sin(\beta)\cos (t)\:\:\sin(\beta)\sin (t)\:\:\cos(\beta)\,\right]^\mathsf{T}$, where $\beta(t)\coloneq \sin(2t)$. The initial state is $\b{q}(0)\coloneq \bs\Lambda(40^\circ,\sqrt2/2[\,-1\:\: 1\:\: 0\,]^\mathsf{T}) \,\b{e}_3$ and the gain is set to $k \coloneq 10$. Given that only $C^1$ continuity is required by the controller, $\psi(\eta)$ is selected as in \eqref{eq:psi} with $n=1$.

Initially, the proposed solution is tested for a static conic inclusion constraint defined by $\b{a} \coloneq \b{e}_3$ and $\alpha_\mathrm{max}(t) \coloneq \pi/6$. 
Fig.~\ref{fig:sphere_const} displays a visualization on the 2-sphere of the tracking results for $\delta = 0.2$. We see that $\b{q}$ converges to $\overbar{\b{q}}_d$, recovering from an initial direction outside of the safe region. After the initial transient, the actual direction never leaves the spherical cap defined by the conic inclusion constraint, correctly tracking the reference whenever inside the safe region, and minimizing the geodesic distance to it otherwise. 
\begin{figure}[t]
	\centering
	\includegraphics[width=0.7\columnwidth]{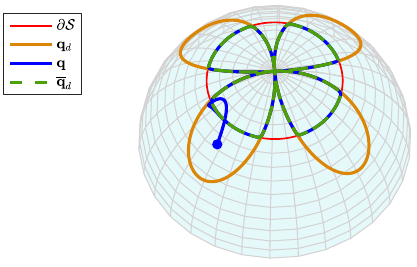} 
	\caption{Tracking under static conic constraint for {\mathversion{normal}$\delta=0.2$}.}
	\label{fig:sphere_const}
\end{figure}
These observations are supported by the tracking-error evolution shown in Fig. \ref{fig:err_static}, where the error is measured as the geodesic distance between $\b{q}$ and $\b{q}_d$. For this scenario, the effect of $\delta$ on the tracking error during boundary transitions is minimal.

\begin{figure}[t]
	\centering
	\includegraphics{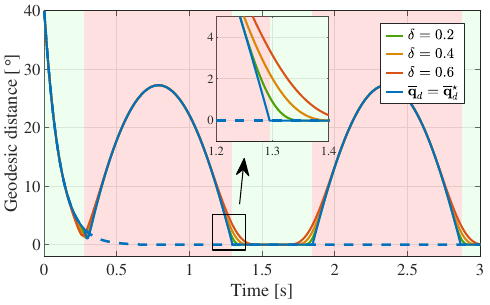} 
	\caption{Tracking error {\mathversion{normal}$d_{\mathbb{S}^2}(\b{q},\b{q}_d)$} for different values of {\mathversion{normal}$\delta$} and limit case {\mathversion{normal}$\overbar{\b{q}}_d=\overbar{\b{q}}_d^\star$}. Regions in green indicate {\mathversion{normal}$d_{\mathbb{S}^2}(\b{q}_d,\b{a})\leq\alpha_\mathrm{max}$} and regions in red indicate {\mathversion{normal}$d_{\mathbb{S}^2}(\b{q}_d,\b{a})>\alpha_\mathrm{max}$}. Dashed line represents {\mathversion{normal}$d_{\mathbb{S}^2}(\b{q},\overbar{\b{q}}_d)$}.}
	\label{fig:err_static}
\end{figure}
\begin{figure}[t!]
	\centering
	\includegraphics{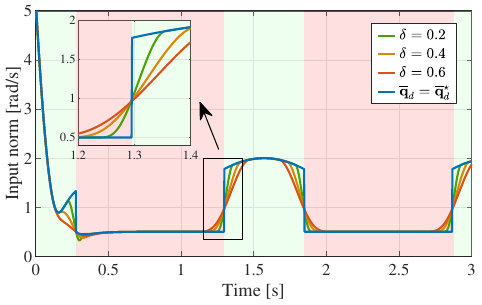} 
	\caption{Input norm, {\mathversion{normal}$\|\b{u}(t)\|$}, for different values of {\mathversion{normal}$\delta$} and for {\mathversion{normal}$\overbar{\b{q}}_d=\overbar{\b{q}}_d^\star$}. }
	\label{fig:input_norm}
\end{figure}
\begin{figure}[t]
	\centering
	\includegraphics[width=0.7\columnwidth]{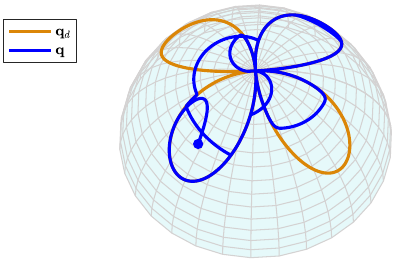} 
	\caption{Tracking under dynamic conic constraint for {\mathversion{normal}$\delta=0.2$}.}
	\label{fig:sphere_var}
\end{figure}
Finally, Fig. \ref{fig:input_norm} shows the time evolution of $\|\b{u}(t)\|$ for different values of $\delta$, together with the non-smooth limit case in which the optimal shaped-reference $\overbar{\b{q}}_d^\star$ is used directly. The results show that the proposed mechanism produces a smooth approximation of the optimal reference, with $\delta$ acting as a tuning parameter that controls the smoothness level.

Under the same reference trajectory, we next impose a dynamic conic inclusion constraint whose axis follows a constant-latitude circle according to $\b{a}(t)\coloneq 0.5\left[\,\cos(t) \:\:\sin(t) \:\:\sqrt{3}\,\right]^\mathsf{T}$, and whose half-angle follows the sinusoidal profile $\alpha_\mathrm{max}(t)\coloneq \pi/6(1.1+\cos(0.5t))$. 

Fig.~\ref{fig:sphere_var} illustrates the tracking results for $\delta = 0.2$. The effect of the dynamic constraint is reflected in the behavior of the actual direction $\b{q}$, which transitions from near-perfect tracking to a more constrained motion dictated by the time-varying cone axis and half-angle. Given the dynamic nature of the problem, the reader is referred to the video available at \href{https://youtu.be/7hQAtjc6k0s}{https://youtu.be/7hQAtjc6k0s} for a better visualization of the results. To further assess the effectiveness of the method, Fig. \ref{fig:alpha} shows the time evolution of $d_{\mathbb{S}^2}(\b{q},\b{a})$ together with $\alpha_\mathrm{\max}(t)$. By overlaying the tracking error $d_{\mathbb{S}^2}(\b{q},\b{q}_d)$, one verifies that the intervals during which the constraint becomes active coincide with those in which the tracking error increases, as expected.

\begin{figure}[t]
	\centering
	\includegraphics[width=1.08\columnwidth]{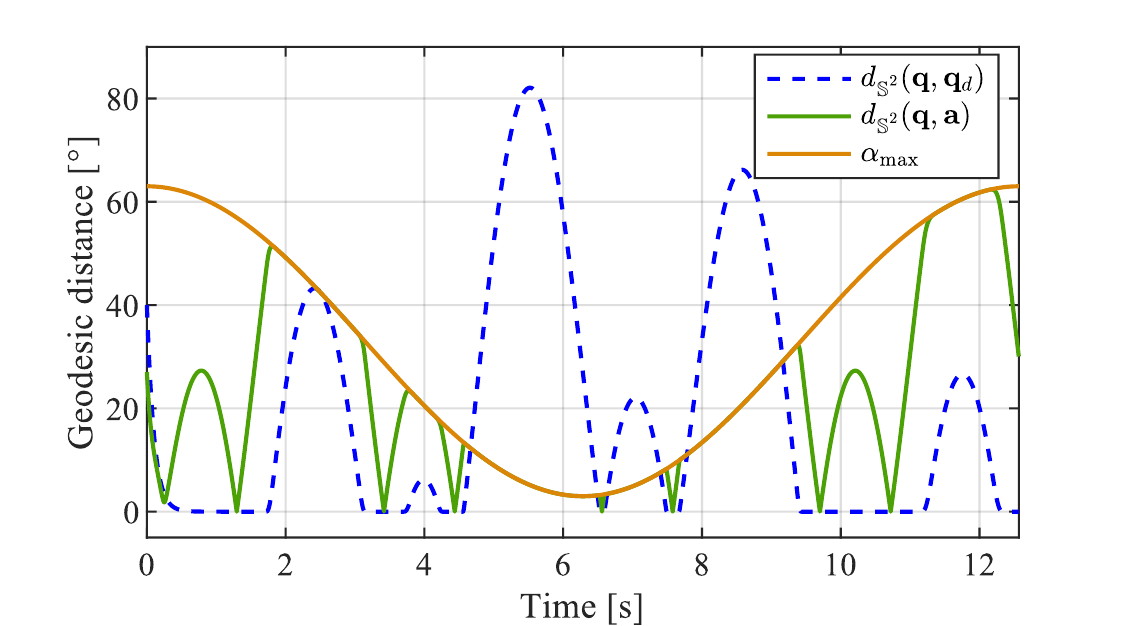} 
	\caption{Time evolution of {\mathversion{normal}$d_{\mathbb{S}^2}(\b{q},\b{q}_d)$},  {\mathversion{normal}$d_{\mathbb{S}^2}(\b{q},\b{a})$}, and {\mathversion{normal}$\alpha_\mathrm{max}$ for $\delta=0.2$}.}
	\label{fig:alpha}
\end{figure}

\section{CONCLUSIONS}

This letter proposed a geometric solution to reduced-attitude tracking on the 2-sphere under a time-varying conic constraint. An unconstrained tracking law was first developed with almost-global asymptotic and regionally exponential stability guarantees. A smooth reference-shaping mechanism was then introduced to generate an admissible reference while preserving the differentiability required by the controller. The resulting closed-form scheme provides continuous feedback, avoids online optimization, and is therefore attractive for real-time implementation on computationally limited platforms. Although the method enforces a relaxed version of the original constraint rather than hard forward invariance of the admissible set, it offers a useful tradeoff in applications where smooth control action and rigorous tracking guarantees are prioritized. Future work will address extension to torque-controlled dynamics and the incorporation of stronger safety guarantees.

\balance
\bibliographystyle{IEEEtran}

\bibliography{refs}

@inproceedings{Sawant25,
	author={Sawant, Mayur and Tayebi, Abdelhamid},
	booktitle={64th IEEE Conf. Dec. Control}, 
	title={Stabilization on the 2-Sphere under Geodesically Convex Constraints}, 
	year={2025},
	volume={},
	number={},
	pages={7401-7406},
	doi={10.1109/CDC57313.2025.11312701}
}

@article{Ibuki2020,
	author={Ibuki, Tatsuya and Wilson, Sean and Ames, Aaron D. and Egerstedt, Magnus},
	journal={IEEE Control Syst. Lett.}, 
	title={Distributed Collision-Free Motion Coordination on a Sphere: A Conic Control Barrier Function Approach}, 
	year={2020},
	volume={4},
	number={4},
	pages={976-981},
	doi={10.1109/LCSYS.2020.2997952}
}

@article{Auntenrieb25,
	author = {Autenrieb, Johannes},
	title = {Quadratic Programming Approach to Flight Envelope Protection Using Control Barrier Functions},
	journal = {J. Guid. Control Dyn.},
	volume = {48},
	number = {11},
	pages = {2622-2633},
	year = {2025},
	doi = {10.2514/1.G009203}
}

@inproceedings{Wu15,
	author={Wu, Guofan and Sreenath, Koushil},
	booktitle={Amer. Control Conf.}, 
	title={Safety-critical and constrained geometric control synthesis using control {Lyapunov} and control Barrier functions for systems evolving on manifolds}, 
	year={2015},
	volume={},
	number={},
	pages={2038-2044},
	doi={10.1109/ACC.2015.7171033}
}

@article{Bullo95,
	title = {Control on the Sphere and Reduced Attitude Stabilization},
	journal = {IFAC Proc. Vol.},
	volume = {28},
	number = {14},
	pages = {495-501},
	year = {1995},
	issn = {1474-6670},
	doi = {https://doi.org/10.1016/S1474-6670(17)46878-9},
	author = {F. Bullo and R.M. Murray and A. Sarti}
}

@article{Dimos21,
	title = {Constrained stabilization on the {$n$}-sphere},
	journal = {Automatica},
	volume = {125},
	year = {2021},
	issn = {0005-1098},
	doi = {https://doi.org/10.1016/j.automatica.2020.109416},
	author = {Soulaimane Berkane and Dimos V. Dimarogonas}
}

@article{Lee17,
	title = {Constrained geometric attitude control on {SO(3)}},
	journal = {Int. J. Control Autom. Syst.},
	volume = {15},
	year = {2017},
	doi = {https://doi.org/10.1007/s12555-016-0607-4},
	author = {Shankar Kulumani and Taeyoung Lee}
}

@article{Romdlony16,
	title = {Stabilization with guaranteed safety using Control {Lyapunov}–Barrier Function},
	journal = {Automatica},
	volume = {66},
	pages = {39-47},
	year = {2016},
	issn = {0005-1098},
	doi = {https://doi.org/10.1016/j.automatica.2015.12.011},
	author = {Muhammad Zakiyullah Romdlony and Bayu Jayawardhana}
}

@ARTICLE{Mesbahi14,
	author={Lee, Unsik and Mesbahi, Mehran},
	journal={IEEE Trans. Aerosp. Electron. Syst.}, 
	title={Feedback control for spacecraft reorientation under attitude constraints via convex potentials}, 
	year={2014},
	volume={50},
	number={4},
	pages={2578-2592},
	doi={10.1109/TAES.2014.120240}
}

@INPROCEEDINGS{Tan20,
	author={Tan, Xiao and Dimarogonas, Dimos V.},
	booktitle={59th IEEE Conf. Dec. Control}, 
	title={Construction of control barrier function and {$\mathcal{C}^2$} reference trajectory for constrained attitude maneuvers}, 
	year={2020},
	pages={3329-3334},
	doi={10.1109/CDC42340.2020.9304324}
}

@article{Fiori24,
	title = {Modeling, simulation and control of a spacecraft: Automated reorientation under directional constraints},
	journal = {Acta Astronaut.},
	volume = {216},
	pages = {214-228},
	year = {2024},
	doi = {https://doi.org/10.1016/j.actaastro.2023.12.053},
	author = {Simone Fiori and Luca Sabatini and Francesco Rachiglia and Edoardo Sampaolesi}
}

@article{Liu23,
	author={Liu, Yueyang and Hu, Qinglei and Feng, Gang},
	journal={IEEE Trans. Aerosp. Electron. Syst.}, 
	title={Adaptive Reduced Attitude Control for Rigid Spacecraft With Elliptical Pointing Constraints}, 
	year={2023},
	volume={59},
	number={4},
	pages={3835-3847},
	doi={10.1109/TAES.2022.3232319}
}

@ARTICLE{Nicotra20,
	author={Nicotra, Marco M. and Liao-McPherson, Dominic and Burlion, Laurent and Kolmanovsky, Ilya V.},
	journal={IEEE Trans. Autom. Control}, 
	title={Spacecraft Attitude Control With Nonconvex Constraints: An Explicit Reference Governor Approach}, 
	year={2020},
	volume={65},
	number={8},
	pages={3677-3684},
	doi={10.1109/TAC.2019.2951303}
}

@article{Shen18,
	title = {Rigid-body attitude stabilization with attitude and angular rate constraints},
	journal = {Automatica},
	volume = {90},
	pages = {157-163},
	year = {2018},
	issn = {0005-1098},
	doi = {https://doi.org/10.1016/j.automatica.2017.12.029},
	author = {Qiang Shen and Chengfei Yue and Cher Hiang Goh and Baolin Wu and Danwei Wang}
}

@article{Ramos18,
	author = {Diaz Ramos, Manuel and Schaub, Hanspeter},
	title = {Kinematic Steering Law for Conically Constrained Torque-Limited Spacecraft Attitude Control},
	journal = {J. Guid. Control Dyn.},
	volume = {41},
	number = {9},
	pages = {1990-2001},
	year = {2018},
	doi = {10.2514/1.G002873}
}

@article{Szmuk20,
	author = {Szmuk, Michael and Reynolds, Taylor P. and A\c{c}\i{}kme\c{s}e, Beh\c{c}et},
	title = {Successive Convexification for Real-Time Six-Degree-of-Freedom Powered Descent Guidance with State-Triggered Constraints},
	journal = {J. Guid. Control Dyn.},
	volume = {43},
	number = {8},
	pages = {1399-1413},
	year = {2020},
	doi = {10.2514/1.G004549}
}

@article{Casau19,
	title = {Robust global exponential stabilization on the n-dimensional sphere with applications to trajectory tracking for quadrotors},
	journal = {Automatica},
	volume = {110},
	year = {2019},
	issn = {0005-1098},
	doi = {https://doi.org/10.1016/j.automatica.2019.108534},
	author = {Pedro Casau and Christopher G. Mayhew and Ricardo G. Sanfelice and Carlos Silvestre}
}

@INPROCEEDINGS{Nakano2018,
	author={Nakano, Satoshi and Nguyen, Tam W. and Garone, Emanuele and Ibuki, Tatsuya and Sampei, Mitsuji},
	booktitle={57th IEEE Conf. Dec. Control}, 
	title={Attitude Constrained Control on {SO(3)}: An Explicit Reference Governor Approach}, 
	year={2018},
	volume={},
	number={},
	pages={1833-1838},
	doi={10.1109/CDC.2018.8618908}
}

@INPROCEEDINGS{Leomanni2024,
	author={Leomanni, Mirko and Dionigi, Alberto and Ferrante, Francesco and Valigi, Paolo and Costante, Gabriele},
	booktitle={63rd IEEE Conf. Dec. Control}, 
	title={Almost Global Trajectory Tracking for Quadrotors Using Thrust Direction Control on {$\mathbb{S}^2$}}, 
	year={2024},
	volume={},
	number={},
	pages={7516-7521},
	doi={10.1109/CDC56724.2024.10886381}
}

@ARTICLE{Serra2016,
		author={Serra, Pedro and Cunha, Rita and Hamel, Tarek and Cabecinhas, David and Silvestre, Carlos},
		journal={IEEE Trans. Robot.}, 
		title={Landing of a Quadrotor on a Moving Target Using Dynamic Image-Based Visual Servo Control}, 
		year={2016},
		volume={32},
		number={6},
		pages={1524-1535},
		doi={10.1109/TRO.2016.2604495}
}

@INPROCEEDINGS{Guiggiani2015,
	author={Guiggiani, Alberto and Kolmanovsky, Ilya and Patrinos, Panagiotis and Bemporad, Alberto},
	booktitle={Proc Amer. Control Conf.}, 
	title={Fixed-point constrained Model Predictive Control of spacecraft attitude}, 
	year={2015},
	volume={},
	number={},
	pages={2317-2322},
	doi={10.1109/ACC.2015.7171078}
}

@ARTICLE{Li2025,
	author={Li, Dongyu and Tong, Shang and Yang, Haoyang and Hu, Qinglei},
	journal={IEEE/ASME Trans. Mechatronics}, 
	title={Time-Synchronized Control for Spacecraft Reorientation With Time-Varying Constraints}, 
	year={2025},
	volume={30},
	number={3},
	pages={2073-2083},
	doi={10.1109/TMECH.2024.3430953}
}

@ARTICLE{Hu2019,
	author={Hu, Qinglei and Chi, Biru and Akella, Maruthi R.},
	journal={IEEE/ASME Trans. Mechatronics}, 
	title={Reduced Attitude Control for Boresight Alignment With Dynamic Pointing Constraints}, 
	year={2019},
	volume={24},
	number={6},
	pages={2942-2952},
	doi={10.1109/TMECH.2019.2944431}
}

@article{He2022,
	title = {A pointing-based method for spacecraft attitude maneuver path planning under time-varying pointing constraints},
	journal = {Adv. Space Res.},
	volume = {70},
	number = {4},
	pages = {1062-1077},
	year = {2022},
	issn = {0273-1177},
	doi = {https://doi.org/10.1016/j.asr.2022.05.058},
	author = {Hanqing He and Peng Shi and Yushan Zhao}
}

@inproceedings{Kim2004,
	author = {Yoonsoo Kim and Mehran Mesbahi and Gurkirpal Singh and Fred Hadaegh},
	title = {On the Constrained Attitude Control Problem},
	booktitle = {AIAA Guid. Navig. Control Conf. Ex.},
	year = {2004},
	chapter = {},
	pages = {},
	doi = {10.2514/6.2004-5129},
}

@article{Tsiotras1994,
	title = {Spin-axis stabilization of symmetric spacecraft with two control torques},
	journal = {Syst. {\&} Control Lett.},
	volume = {23},
	number = {6},
	pages = {395-402},
	year = {1994},
	issn = {0167-6911},
	doi = {https://doi.org/10.1016/0167-6911(94)90093-0},
	author = {Panagiotis Tsiotras and James M. Longuski}
}

@article{Pong2015,
	author = {Pong, Christopher M. and Miller, David W.},
	title = {Reduced-Attitude Boresight Guidance and Control on Spacecraft for Pointing, Tracking, and Searching},
	journal = {J. Guid. Control Dyn.},
	volume = {38},
	number = {6},
	pages = {1027-1035},
	year = {2015},
	doi = {10.2514/1.G000264}
}

@Article{Coates2021,
	AUTHOR = {Coates, Erlend M. and Fossen, Thor I.},
	TITLE = {Geometric Reduced-Attitude Control of Fixed-Wing {UAVs}},
	JOURNAL = {Appl. Sci.},
	VOLUME = {11},
	YEAR = {2021},
	NUMBER = {7},
	ARTICLE-NUMBER = {3147},
	ISSN = {2076-3417},
	DOI = {10.3390/app11073147}
}

@INPROCEEDINGS{Lee2016,
	author={Lee, Taeyoung},
	booktitle={55th IEEE Conf. Dec. Control}, 
	title={Optimal hybrid controls for global exponential tracking on the two-sphere}, 
	year={2016},
	volume={},
	number={},
	pages={3331-3337},
	doi={10.1109/CDC.2016.7798770}}

@book{Ebert2002,
	title={Texturing and Modeling: A Procedural Approach},
	author={Ebert, D.S. and Musgrave, F.K. and Peachey, D. and Perlin, K. and Worley, S.},
	isbn={9780080518756},
	year={2002},
	publisher={Elsevier Science}
}

\end{document}